\documentclass[11pt,a4paper,final]{amsart}
\usepackage[cp1250]{inputenc}
\usepackage[T1]{fontenc}
\usepackage{amsfonts}
\usepackage{amsmath}
\usepackage{amssymb}
\usepackage{amsthm}
\usepackage{mathrsfs}
\usepackage{amsopn}
\usepackage{indentfirst}
\usepackage{enumerate}
\usepackage{graphicx}
\usepackage{hyperref}

\hypersetup{
  colorlinks=true,
  linkcolor=blue,
  citecolor=blue,
  urlcolor=blue,
  pdftitle={A Note on the Classification of Centroaffine Hypersurfaces}
}

\newtheoremstyle{definition}{8pt}{}{}{0cm}{\bf}{.\ }{0pt}{}
\numberwithin{equation}{section}
\newtheorem{theorem}{Theorem}[section]
\newtheorem{lemma}[theorem]{Lemma}
\newtheorem{proposition}[theorem]{Proposition}

\theoremstyle{definition}
\newtheorem{definition}[theorem]{Definition}

\newcommand{\R}{\mathbb{R}}

\newcommand{\X}{\mathcal{X}}

\newcommand{\const}{\operatorname{const}}

\newcommand{\id}{\operatorname{id}}
\newcommand{\Span}{\operatorname{span}}

\newcommand{\tr}{\operatorname{tr}}

\newcommand{\cA}{\mathcal{A}}
\newcommand{\epsc}{\widetilde\varepsilon}

\title{A Note on the Classification of Centroaffine Hypersurfaces}

\author[Z. Szancer]{Zuzanna Szancer}
\address{Zuzanna Szancer, Department of Applied Mathematics, University of Agriculture in Krakow, 253 Balicka St., 30-198 Krakow, Poland}
\email{Zuzanna.Szancer@urk.edu.pl}


\begin{document}

\begin{abstract}
Cao and Wang recently classified a class of locally strongly convex centroaffine hypersurfaces whose non-vanishing Tchebychev vector field $T$ satisfies
$K(T,X)=\alpha X$ and $X(\alpha)=0$ for every $X\perp T$, together with $K(T,T)=\beta T$.
In the degenerate case $\beta=2\alpha$, their higher-dimensional argument assumes in addition that the $T^\perp$-component of the difference tensor $K$ vanishes identically.
We show that this extra assumption is unnecessary and the theorem can be strengthened.
\end{abstract}

\maketitle

\section{Introduction}

Let $x\colon M^n\to\R^{n+1}$ be a locally strongly convex centroaffine hypersurface, $n\ge3$, with positive definite centroaffine metric $h$, induced connection $\nabla$, Levi-Civita connection $\widehat\nabla$, and difference tensor
$$
K(X,Y)=\nabla_XY-\widehat\nabla_XY.
$$
The Tchebychev vector field is denoted by $T$.
Cao and Wang \cite{CaoWang2026} considered hypersurfaces with $T\neq0$ satisfying
\begin{equation}\label{eq::main-condition}
K(T,X)=\alpha X, \qquad X(\alpha)=0 \quad\text{for every }X\perp T,
\end{equation}
and
\begin{equation}\label{eq::betaCondition}
K(T,T)=\beta T,
\end{equation}
where $\alpha$ and $\beta$ are smooth functions.
\par Their main result shows that, under some additional assumptions, $(M,h)$ is locally isometric to a warped product of an interval and an $(n-1)$-manifold, and that the immersion can be written in terms of a curve and a proper or improper affine hypersphere.
When $\beta = 2\alpha$ and $n\ge4$ Cao and Wang assume that the restriction of the difference tensor to the orthogonal distribution $T^\perp$ has vanishing $T^\perp$-component.  In the notation used below, this is the condition $\bar K=0$. The purpose of this paper is to show that this condition is not needed.
Thus we can state a strengthened version of the classification theorem as follows.
\begin{theorem}[Strengthened Cao--Wang classification]\label{tw::main}
Let $x\colon M^n\to\R^{n+1}$, $n\ge3$, be a locally strongly convex centroaffine hypersurface with non-vanishing Tchebychev vector field $T$ satisfying \eqref{eq::main-condition} and \eqref{eq::betaCondition}.  Then $(M,h)$ is locally isometric to a warped product
$$
I\times_\rho N.
$$
Moreover, after a centroaffine transformation, the immersion is locally of the form
\begin{equation}\label{eq::explicit-immersion-intro}
x(t,p)=\gamma_1(t)e+\gamma_2(t)g(p),\qquad \gamma_1=\kappa^{-1}\phi, \quad \gamma_2=\kappa^{-1},
\end{equation}
where $e=(0,\ldots,0,1)$ and $g$ is either a proper affine hypersphere in $\R^n\times\{0\}$ or an improper affine hypersphere in $\{1\}\times\R^n$.
The functions $\kappa,\tau,\phi$ satisfy
\begin{equation}\label{eq::ODE-intro}
\kappa'=-(b+\lambda)\kappa, \qquad \tau'=(\lambda+b-\mu)\tau, \qquad \phi'=\kappa\tau^{-1},
\end{equation}
where
$$
\lambda=\alpha\|T\|^{-1},\qquad \mu=\beta\|T\|^{-1},
$$
and
\begin{equation}\label{eq::ODE-lambda-b-intro}
\lambda'=(\mu-2\lambda)b, \qquad b'=-b^2-\epsc+\mu\lambda-\lambda^2.
\end{equation}
In particular, no restriction on $\bar K$ is required when $\beta=2\alpha$ and $n\ge 4$.
\end{theorem}
\par The paper is organized as follows.  Section~\ref{sec::prelim} fixes the notation and collects in one lemma the core identities already obtained by Cao and Wang in the case $\beta=2\alpha$.  Section~\ref{sec::degenerate} contains main results of the paper. It proves  strengthened version of Proposition~4, and then completes the proof of the classification Theorem \ref{tw::main}.

\section{Preliminaries}\label{sec::prelim}
We use the conventions of \cite{CaoWang2026} as well as \cite{NomSas}.
Let $x\colon M^n\to\R^{n+1}$ be a locally strongly convex centroaffine hypersurface, $n\ge3$, with positive definite centroaffine metric $h$ and induced connection $\nabla$.
Let $\widehat\nabla$ be the Levi-Civita connection for $h$.
We have the centroaffine Gauss formula
\begin{equation}\label{eq::centro-gauss-formula}
D_Xx_*Y=x_*(\nabla_XY)- h(X,Y)\epsc x, \qquad \epsc\in\{\pm1\}
\end{equation}
and define a difference tensor $K$ by the formula
$$
K(X,Y)=\nabla_XY-\widehat\nabla_XY.
$$
The standard centroaffine cubic form is defined by the formula
\begin{equation}\label{eq::cubicForm}
C(X,Y,Z):=(\nabla_Xh)(Y,Z).
\end{equation}
$C$ is totally symmetric and is related to $K$ by

$$
C(X,Y,Z)=(\nabla_Xh)(Y,Z)=-2h(K(X,Y),Z),
$$
for all $X,Y,Z\in\X(M)$. Recall that Tchebychev vector field $T$ is defined by the formula
$$
T:=\frac{1}{n}\operatorname{tr}_h K.
$$
For a local $h$-orthonormal frame $e_1,\ldots,e_n$, this is equivalent to
$$
T=\frac{1}{n}\sum_{i=1}^n K(e_i,e_i).
$$
In this paper we always assume that $T$ is non-zero everywhere.
\par Let us define
$$
e_1:=\frac{T}{\|T\|},\qquad T_1:=\Span\{e_1\}, \qquad T_2:=T_1^\perp.
$$
The vector field $e_1$ is normalized $T$. Whenever $e_2,\ldots,e_n$ are used below, they denote an arbitrary local $h$-orthonormal frame of $T_2$, thus $\{e_1,e_2,\ldots,e_n\}$ is a local orthonormal frame of $TM$.
Set
\begin{equation}\label{eq::lambda-mu}
\mu=\beta\|T\|^{-1},\qquad\lambda=\alpha\|T\|^{-1}.
\end{equation}
Then
\begin{equation}\label{eq::basic-K}
K(e_1,e_1)=\mu e_1, \qquad K(e_1,X)=\lambda X, \qquad X\in T_2,
\end{equation}
and
\begin{equation}\label{eq::T-length}
\|T\|=\frac1n\bigl(\mu+(n-1)\lambda\bigr)\neq0.
\end{equation}

For $X,Y\in T_2$, total symmetry of \eqref{eq::cubicForm} gives
$$
h(K(X,Y),e_1)=h(K(e_1,X),Y)=\lambda h(X,Y).
$$
Therefore we have decomposition of $K$ as follows:
\begin{equation}\label{eq::Kdecomp}
K(X,Y)=\bar K(X,Y)+\lambda h(X,Y)e_1,\qquad X,Y\in T_2,
\end{equation}
where $\bar K$ is component of $T_2$.
For $Z\in T_2$ we denote
\begin{equation}\label{eq::barKZ}
\bar K_ZX:=\bar K(Z,X).
\end{equation}
We also define
\begin{equation}\label{eq::Adef}
\cA X:=-\widehat\nabla_Xe_1,\qquad X\in T_2.
\end{equation}
Since $h(e_1,e_1)=1$,
$$
h(\widehat\nabla_Xe_1,e_1)=0,
$$
so $\cA$ maps $T_2$ into $T_2$.  Moreover, for $X,Y\in T_2$,
$$
0=X(h(Y,e_1))=h(\widehat\nabla_XY,e_1)+h(Y,\widehat\nabla_Xe_1).
$$
Therefore
\begin{equation}\label{eq::A-second-form}
  h(\cA X,Y)=h(\widehat\nabla_XY,e_1),
  \qquad X,Y\in T_2.
\end{equation}
For two endomorphisms $S$ and $L$ of the same vector space, we write
$$
  [S,L]:=S\circ L-L\circ S.
$$
The following lemma contains results obtained by Cao and Wang in \cite{CaoWang2026} that are important for the proof of Theorem \ref{tw::main}.  The equation numbers quoted in its proof are the equation numbers of \cite{CaoWang2026}; they are written explicitly and do not depend on the numbering of the present paper.
\begin{lemma}[Cao--Wang, \cite{CaoWang2026}]\label{lm::CW}
Assume \eqref{eq::main-condition} and \eqref{eq::betaCondition}, with $T\neq0$, and suppose
$$
\beta=2\alpha.
$$
Then
\begin{equation}\label{eq::start-lambda}
\mu=2\lambda, \qquad \lambda=\const,\qquad \lambda>0,
\end{equation}
and
\begin{equation}\label{eq::geodesic}
\widehat\nabla_{e_1}e_1=0.
\end{equation}
The distribution $T_2$ is integrable.  Moreover, the tensor $\bar K$ defined in \eqref{eq::Kdecomp} satisfies
\begin{equation}\label{eq::barK-symmetry}
h(\bar K(X,Y),Z)=h(\bar K(Y,Z),X)=h(\bar K(Z,X),Y),
\qquad X,Y,Z\in T_2,
\end{equation}
and
\begin{equation}\label{eq::barK-tracefree}
\sum_{i=2}^n\bar K(e_i,e_i)=0
\end{equation}
for any local $h$-orthonormal frame $e_2,\ldots,e_n$ of $T_2$.  Finally, $\cA$ defined in \eqref{eq::Adef} is $h$-self-adjoint and
\begin{equation}\label{eq::A-K-commute}
[\cA,\bar K_Z]=0\qquad\text{for every }Z\in T_2.
\end{equation}
\end{lemma}
\begin{proof}
From $\beta=2\alpha$ one has $\mu=2\lambda$.
Equation (2.9) from \cite{CaoWang2026} gives $e_1(\lambda)=0$, while equation (2.15) from \cite{CaoWang2026} gives
$$
e_i(\lambda)=0,\qquad i=2,\ldots,n,
$$
thus $\lambda=\const$.  Since the equation (2.2) from \cite{CaoWang2026} becomes
$$
\|T\|=\frac{n+1}{n}\lambda,
$$
the assumption $T\neq0$ implies $\lambda>0$. From their equation (2.11) we obtain \eqref{eq::geodesic} from this paper.
The equation (2.17) from \cite{CaoWang2026} states
$$
\widehat\Gamma^1_{ij}=\widehat\Gamma^1_{ji}, \qquad i\neq j,\quad i,j=2,\ldots,n.
$$
Hence
$$
h([e_i,e_j],e_1)=\widehat\Gamma^1_{ij}-\widehat\Gamma^1_{ji}=0,
$$
so $T_2$ is integrable.  By \eqref{eq::A-second-form} from this paper, the same identity shows that $\cA$ is self-adjoint.
Equation \eqref{eq::barK-symmetry} from this paper follows from the total symmetry of the centroaffine cubic form \eqref{eq::cubicForm} from this paper.
Since $e_1=\frac{T}{\|T\|}$, $T=\frac{1}{n}\operatorname{tr}_h K$ and $K(e_1,e_1)=\mu e_1=2\lambda e_1$ we easily obtain \eqref{eq::barK-tracefree} from this paper (see (2.4) in \cite{CaoWang2026}).
Finally, using \eqref{eq::A-second-form} from this paper and the definition of $\bar K$, Cao--Wang's equation (2.16) and (2.17) from \cite{CaoWang2026} give \eqref{eq::A-K-commute} from this paper.
\end{proof}
We shall use the following standard local characterization of warped products due to Hiepko \cite{Hiepko1979}, see also \cite{Chen2017}.
\begin{definition}\label{def::spherical}
Let $\mathcal D$ be an integrable distribution on a Riemannian manifold $(M,h)$, and let $\mathcal D^\perp$ denote its orthogonal complement.  We call $\mathcal D$ \emph{spherical} if every integral manifold $N$ of $\mathcal D$ is totally umbilical and its mean curvature vector is parallel in the normal bundle.  Equivalently, there is a normal vector field $H\in\Gamma(\mathcal D^\perp)$ along each leaf such that, for all $X,Y\in\Gamma(\mathcal D)$,
$$
\widehat\nabla_XY=\nabla^N_XY+h(X,Y)H,\qquad\widehat\nabla_XH\in\Gamma(\mathcal D).
$$
Here $\nabla^N$ is the Levi--Civita connection of the metric induced on the leaf $N$.  Thus the first identity says that the second fundamental form of $N\subset M$ is $h(X,Y)H$, while the second says that the normal component of $\widehat\nabla_XH$ vanishes.
\end{definition}

\begin{theorem}[Hiepko, \cite{Hiepko1979}]\label{tw::Hiepko}
Let $(M,h)$ be a Riemannian manifold and suppose that
$$
TM=\mathcal D_1\oplus\mathcal D_2
$$
is an orthogonal decomposition into complementary integrable distributions.  If $\mathcal D_1$ is totally geodesic and $\mathcal D_2$ is spherical, then every point of $M$ has a neighborhood isometric to a warped product
$$
M_1\times_\rho M_2,
$$
where $M_i$ is an integral manifold of $\mathcal D_i$.
\end{theorem}

\section{Main Results}\label{sec::degenerate}

The main purpose of this section is to prove strengthened version of Proposition 4 from \cite{CaoWang2026} and in consequence Theorem \ref{tw::main}. First we shall prove the following
\begin{lemma}\label{lm::projection}
Let $(N^m,h)$ be a Riemannian manifold and let
$$
B\colon TN\longrightarrow TN
$$
be a smooth $h$-self-adjoint endomorphism.
Then for every $p\in N$ and $c\in\mathbb R\setminus\operatorname{Spec}(B_p)$
there exists a neighborhood $U$ of $p$ such that
$$
c\notin\operatorname{Spec}(B_q)\qquad(q\in U),
$$
and the spaces
$$
E_{>c}(q):=\bigoplus_{\substack{\xi\in\operatorname{Spec}(B_q)\\ \xi>c}}\ker(B_q-\xi\id)\subset T_qN
$$
have constant dimension and form a smooth vector subbundle
$$
  E_{>c}\subset TN|_U.
$$
Let $P$ be the $h$-orthogonal projection onto $E_{>c}$.  Then $P$ is smooth and
$$
[P,B]=0.
$$
Moreover, if a smooth endomorphism $L\colon TN\to TN$ satisfies $[B,L]=0$, then $[P,L]=0$.
\end{lemma}
\begin{proof}
Let $p\in N$, $c\in\mathbb R\setminus\operatorname{Spec}(B_p)$. Since $B$ is smooth, there exists a neighborhood $U$ of $p$ such that $c\notin\operatorname{Spec}(B_q)$ for all $q\in U$.
Let
$$
\xi_1(q)\le\cdots\le\xi_m(q)
$$
denote the ordered eigenvalues of $B_q$, counted with multiplicity.  Since $B$ is smooth and self-adjoint, the functions $\xi_i$ are continuous (but not necessarily smooth).  Since $c\notin\operatorname{Spec}(B_p)$,
$$
\delta:=\min_{1\le i\le m}|\xi_i(p)-c|>0.
$$
Shrinking $U$ if needed, we obtain
$$
|\xi_i(q)-\xi_i(p)|<\frac{\delta}{2}
$$
for every $q\in U$ and every $i=1,\ldots,m$.
Thus no eigenvalue crosses the level $c$ on $U$, and in consequence $\dim E_{>c}(q)$ is constant.  Denote this dimension by $r$.
If $r=0$ or $r=m$, the assertion is immediate.  Hence assume $1\le r\le m-1$.
\par At $p$ we define subspaces
$$
V_p^+:=E_{>c}(p),\qquad V_p^-:=(V_p^+)^\perp.
$$
Choose a smooth local $h$-orthonormal frame $v_1,\ldots,v_m$ near $p$ such that
$$
v_1(p),\ldots,v_r(p)
$$
is an orthonormal basis of $V_p^+$ and
$$
v_{r+1}(p),\ldots,v_m(p)
$$
is an orthonormal basis of $V_p^-$.  Let $V$ and $W$ be the local subbundles spanned by the first $r$ and the last $m-r$ vectors of this frame, respectively.  Thus
$$
TN|_U=V\oplus W.
$$
With respect to this decomposition write
$$
B=
\begin{pmatrix}
B_{11} & B_{12}\\
B_{21} & B_{22}
\end{pmatrix},
$$
where
$$
B_{11}\colon V\to V,\qquad B_{12}\colon W\to V, \qquad B_{21}\colon V\to W, \qquad B_{22}\colon W\to W.
$$
At $p$ one has
$$
B_{12}(p)=B_{21}(p)=0,
$$
since $V_p^+$ and $V_p^-$ are $B_p$-invariant.  Moreover, every eigenvalue of $B_{11}(p)$ is larger than $c$, whereas every eigenvalue of $B_{22}(p)$ is smaller than $c$.
Using the chosen local frame, we identify each space $\operatorname{Hom}(V_q,W_q)$ with the fixed vector space $\mathbb R^{(m-r)\times r}$.  Define
$$
F\colon U\times\mathbb R^{(m-r)\times r}\longrightarrow\mathbb R^{(m-r)\times r}
$$
by
\begin{equation}\label{eq::spectral-F}
F(q,X)=B_{21}(q)+B_{22}(q)X-XB_{11}(q)-XB_{12}(q)X.
\end{equation}
For $v\in V_q$,
$$
B_q(v+Xv)=\bigl(B_{11}(q)v+B_{12}(q)Xv\bigr)+\bigl(B_{21}(q)v+B_{22}(q)Xv\bigr).
$$
Hence the graph of $X$ is $B_q$-invariant if and only if
$$
B_{21}(q)v+B_{22}(q)Xv=X\bigl(B_{11}(q)v+B_{12}(q)Xv\bigr)
$$
for every $v\in V_q$, equivalently if and only if
$$
F(q,X)=0.
$$
Since $B_{21}(p)=0$,
$$
F(p,0)=0.
$$
The derivative of $F$ with respect to the second variable at $(p,0)$ is
$$
D_XF(p,0)[Y]=B_{22}(p)Y-YB_{11}(p),\qquad Y\in\mathbb R^{(m-r)\times r}.
$$
This linear map is invertible.  Indeed, choose orthonormal eigenbases
$$
B_{11}(p)u_i=\alpha_i u_i, \qquad B_{22}(p)w_j=\beta_j w_j,
$$
where $\alpha_i>c$ and $\beta_j<c$.  If
$$
Yu_i=\sum_j y_{ji}w_j,
$$
then
$$
\bigl(B_{22}(p)Y-YB_{11}(p)\bigr)u_i=\sum_j(\beta_j-\alpha_i)y_{ji}w_j.
$$
Since $\beta_j-\alpha_i\neq0$ for all $i,j$, the kernel is zero.  The domain and codomain have the same dimension $r(m-r)$, so $D_XF(p,0)$ is invertible.
After choosing local coordinates on $U$, the smooth implicit function theorem gives, after reducing $U$ if necessary, a unique smooth map
$$
X\colon U\longrightarrow\mathbb R^{(m-r)\times r}
$$
such that
$$
X(p)=0,\qquad F(q,X(q))=0
$$
for every $q\in U$.  Interpreting $X(q)$ again as a linear map $X(q)\colon V_q\to W_q$, its graphs
$$
G_q:=\{v+X(q)v:\ v\in V_q\}
$$
form a smooth rank-$r$ subbundle $G\subset TN|_U$ invariant under $B$.
At $p$ we have
$$
G_p=V_p^+=E_{>c}(p).
$$
In the smooth frame $v_i+Xv_i$, $1\le i\le r$, the restriction $B|_G$ is represented by
$$
B_{11}+B_{12}X.
$$
At $p$ all eigenvalues of $B|_{G_p}$ are larger than $c$.  Since $B|_G$ is a smooth self-adjoint endomorphism of the smooth bundle $G$, its ordered eigenvalues are continuous.
After reducing $U$ once more, all eigenvalues of $B_q|_{G_q}$ are therefore larger than $c$ for every $q\in U$.  Hence
$$
G_q\subset E_{>c}(q).
$$
Since
$$
\dim G_q=r=\dim E_{>c}(q),
$$
we obtain
$$
G_q=E_{>c}(q)\qquad(q\in U).
$$
Thus $E_{>c}$ is a smooth vector subbundle of $TN|_U$.
\par Let $P$ be the $h$-orthogonal projection onto $E_{>c}$.  Since $E_{>c}$ is smooth, $P$ is a smooth tensor field of type $(1,1)$.
Both $E_{>c}(q)$ and its orthogonal complement are sums of eigenspaces of the self-adjoint endomorphism $B_q$, so they are preserved by $B_q$.  Hence $[P,B]=0$.
If $[B,L]=0$, then $L_q$ preserves every eigenspace of $B_q$, and therefore preserves both $E_{>c}(q)$ and its orthogonal complement.  Thus $[P,L]=0$.
\end{proof}

We now apply this projection to the endomorphism $\cA$ itself.  By Lemma~\ref{lm::CW}, every local integral leaf $N$ of $T_2$ carries the induced Riemannian metric. Let $\bar\nabla$ denote the Levi-Civita connection of this induced metric.

\begin{lemma}\label{lm::operatorA}
Assume the setting of Lemma~\ref{lm::CW}.
Let $N$ be a local integral leaf of $T_2$, with Levi-Civita
connection $\bar\nabla$.
Suppose that, for every $q\in N$, every $X\in T_qN$, and any
$h$-orthonormal basis $e_2,\ldots,e_n$ of $T_qN$, one has
\begin{equation}\label{eq::Kbar}
\sum_{i=2}^n(\bar\nabla_{e_i}\bar K)(e_i,X)=-\lambda\bigl((n-1)\cA-(\tr\cA)\id\bigr)X,
\end{equation}
where $\lambda\neq0$, all terms in
\eqref{eq::Kbar} are evaluated at $q$, and
$\tr\cA$ denotes the fiberwise trace of $\cA|_{TN}$. Then
$$
\cA=\frac{\tr\cA}{n-1}\id\qquad\text{on }TN.
$$
\end{lemma}
\begin{proof}
Assume that $\cA_p$ is not a scalar endomorphism at some point $p\in N$.  Since $\cA_p$ is self-adjoint, its eigenvalues are real.  Their average is
$$
\frac{\tr\cA_p}{n-1}.
$$
Because they are not all equal, at least one eigenvalue is strictly larger than this average.  Choose $c$ such that
\begin{equation}\label{eq::c-choice-A}
\frac{\tr\cA_p}{n-1}<c<\min\left\{\xi\in\operatorname{Spec}(\cA_p):\xi>\frac{\tr\cA_p}{n-1}\right\}.
\end{equation}
By the choice of $c$, we have $c\notin\operatorname{Spec}(\cA_p)$. Now using Lemma~\ref{lm::projection} for $B=\cA$ we get that
there exists a neighborhood $U\subset N$ of $p$ and a smooth $h$-orthogonal projection
$$
P\in\Gamma\bigl(\operatorname{End}(TN|_U)\bigr)
$$
onto the subbundle $E_{>c}$ corresponding to eigenvalues larger than $c$.
At $p$ we have
$$
\operatorname{im}P_p=\bigoplus_{\substack{\xi\in\operatorname{Spec}(\cA_p)\\\xi>\tr\cA_p/(n-1)}}\ker(\cA_p-\xi\id).
$$
Let $\xi_1,\ldots,\xi_r$ be the eigenvalues of the restriction $\cA_p|_{\operatorname{im}P_p}$, counted with multiplicity.  Then
$$
\xi_k>\frac{\tr\cA_p}{n-1},\qquad k=1,\ldots,r,
$$
and therefore
\begin{equation}\label{eq::GreaterThanZero}
\tr\bigl(P_p((n-1)\cA_p-(\tr\cA_p)\id)\bigr)=(n-1)\sum_{k=1}^r\left(\xi_k-\frac{\tr\cA_p}{n-1}\right)>0.
\end{equation}
We shall show that \eqref{eq::Kbar} forces the above trace to vanish.

For every $q\in U$ and $Z\in T_qN$, Lemma~\ref{lm::CW} gives $[\cA_q,\bar K_Z]=0$ thus from Lemma~\ref{lm::projection} we obtain
\begin{equation}\label{eq::PcommutesKbar}
[P_q,\bar K_Z]=0\qquad(q\in U,\ Z\in T_qN).
\end{equation}
Define the one-form $\omega_P\in\Omega^1(U)$ by
\begin{equation}\label{eq::omegaP}
(\omega_P)_q(Z):=\tr(P_q\bar K_Z),\qquad q\in U,\ Z\in T_qN.
\end{equation}
Fix $q\in U$ and let $e_2,\ldots,e_n$ be any $h$-orthonormal basis of $T_qN$.  Self-adjointness of $P_q$ and formulas \eqref{eq::barK-symmetry} and \eqref{eq::PcommutesKbar} give
\begin{align*}
(\omega_P)_q(Z)&=\sum_{j=2}^n h(P_q\bar K_Ze_j,e_j)=\sum_{j=2}^n h(\bar K(Z,e_j),P_qe_j)\\
&=\sum_{j=2}^n h(\bar K(e_j,P_qe_j),Z)=\sum_{j=2}^n h(P_q\bar K(e_j,e_j),Z)=0,
\end{align*}
where the last equality we get from \eqref{eq::barK-tracefree}.  Since $q\in U$ and $Z\in T_qN$ were arbitrary we obtain
\begin{equation}\label{eq::omegazero}
\omega_P\equiv0\qquad\text{on }U.
\end{equation}

After shrinking $U$ if necessary we may choose an $h$-orthonormal frame $e_2,\ldots,e_n$ on $U$ such that
\begin{equation}\label{eq::normalFrame}
(\bar\nabla_{e_i}e_j)_p=0,\quad i,j=2,\ldots,n.
\end{equation}
Since $\omega_P\equiv0$, we have
\begin{align*}
0&=\sum_{i=2}^n e_i\bigl(\omega_P(e_i)\bigr)=\sum_{i,j=2}^n e_i\bigl(h(P\bar K(e_i,e_j),e_j)\bigr)\\
 &=\sum_{i,j=2}^n h\bigl((\bar\nabla_{e_i}P)\bar K(e_i,e_j),e_j\bigr)+\sum_{i,j=2}^n h\bigl(P(\bar\nabla_{e_i}\bar K)(e_i,e_j),e_j\bigr)\\
 &\quad+\sum_{i,j=2}^n h\bigl(P\bar K(\bar\nabla_{e_i}e_i,e_j),e_j\bigr)+\sum_{i,j=2}^n h\bigl(P\bar K(e_i,\bar\nabla_{e_i}e_j),e_j\bigr)+\sum_{i,j=2}^n h\bigl(P\bar K(e_i,e_j),\bar\nabla_{e_i}e_j\bigr).
\end{align*}
Evaluating this identity at $p$, the last three sums vanish by \eqref{eq::normalFrame}.  Hence
\begin{align}\label{eq::omegaDiff}
0=\sum_{i,j=2}^n h\bigl((\bar\nabla_{e_i}P)_p\bar K(e_i,e_j),e_j\bigr)+\sum_{i,j=2}^n h\bigl(P_p(\bar\nabla_{e_i}\bar K)(e_i,e_j),e_j\bigr)
\end{align}
at $p$.
\par We first show that the first sum vanishes.  For $X\in T_pN$, differentiating $P^2=P$ at $p$ in the direction $X$ gives
$$
(\bar\nabla_XP)_pP_p+P_p(\bar\nabla_XP)_p=(\bar\nabla_XP)_p.
$$
Composing on the left with $P_p$ and using $P_p^2=P_p$, we obtain
$$
P_p(\bar\nabla_XP)_pP_p+P_p(\bar\nabla_XP)_p=P_p(\bar\nabla_XP)_p,
$$
hence
$$
P_p(\bar\nabla_XP)_pP_p=0.
$$
Moreover,
\begin{align*}
(\id-P_p)(\bar\nabla_XP)_p(\id-P_p)=(\bar\nabla_XP)_p-(\bar\nabla_XP)_pP_p-P_p(\bar\nabla_XP)_p+P_p(\bar\nabla_XP)_pP_p=0.
\end{align*}
Therefore
$$
(\bar\nabla_XP)_p(\operatorname{im}P_p)\subset\ker P_p,\qquad (\bar\nabla_XP)_p(\ker P_p)\subset\operatorname{im}P_p.
$$
On the other hand, \eqref{eq::PcommutesKbar} at $p$ implies that $\bar K_Z$ preserves both $\operatorname{im}P_p$ and $\ker P_p$ for every $Z\in T_pN$.  Hence $(\bar\nabla_XP)_p\bar K_Z$ maps $\operatorname{im}P_p$ into $\ker P_p$ and $\ker P_p$ into $\operatorname{im}P_p$.  In an orthonormal basis of $T_pN$ such that  first vectors span $\operatorname{im}P_p$ and rest span $\ker P_p$, every diagonal entry of this composition is zero.  Therefore
\begin{equation}\label{eq::trace-derivative-term-zero}
  \tr\bigl((\bar\nabla_XP)_p\bar K_Z\bigr)=0
  \qquad\text{for every }X,Z\in T_pN.
\end{equation}
Applying the above for $X=Z=e_i$ we obtain that the first sum in \eqref{eq::omegaDiff} is zero.
For the second sum, using \eqref{eq::Kbar} at $p$, we get:
\begin{align*}
0&=\sum_{j=2}^n h\left( P_p\left[\sum_{i=2}^n(\bar\nabla_{e_i}\bar K)(e_i,e_j)\right],e_j\right)\\
 &=-\lambda\sum_{j=2}^n h\bigl(P_p((n-1)\cA_p-(\tr\cA_p)\id)e_j,e_j\bigr)\\
 &=-\lambda\tr\bigl(P_p((n-1)\cA_p-(\tr\cA_p)\id)\bigr).
\end{align*}
Since $\lambda>0$, this implies
$$
\tr\bigl(P_p((n-1)\cA_p-(\tr\cA_p)\id)\bigr)=0,
$$
which contradicts the strict positivity established in \eqref{eq::GreaterThanZero}. In consequence no point $p\in N$ can have non-scalar $\cA_p$, and hence
$$
\cA=\frac{\tr\cA}{n-1}\id\qquad\text{on }TN.
$$
\end{proof}
Now we can prove the following strengthened version of Proposition~4 from \cite{CaoWang2026}.
\begin{proposition}[Strengthened Proposition 4]\label{stw::strengthenedProp4}
Let $M^n$, $n\ge3$, be a locally strongly convex centroaffine hypersurface satisfying
$$
K(T,X)=\alpha X,\qquad X(\alpha)=0\quad\text{for every }X\perp T,
$$
where the Tchebychev vector field $T$ is non-zero, and suppose
$$
K(T,T)=\beta T.
$$
If
$$
\beta=2\alpha,
$$
then $(M,h)$ admits locally a warped-product structure
$$
(M,h)=I\times_\rho N.
$$
No restriction on the $T_2$-component $\bar K$ of the difference tensor is required.
\end{proposition}
\begin{proof}
By Lemma~\ref{lm::CW}, $\lambda$ is a positive constant, $T_1$ is totally geodesic, $T_2$ is integrable, and $\cA$ is self-adjoint.
Let $N$ be a local integral leaf of $T_2$, and let $\bar\nabla$ be the Levi-Civita connection of its induced metric.  From \eqref{eq::A-second-form} and integrability of $T_2$ we have
\begin{equation}\label{eq::leaf-connection}
\widehat\nabla_XY=\bar\nabla_XY+h(\cA X,Y)e_1,\qquad X,Y\in T_2.
\end{equation}
First we compute $(\widehat\nabla_XK)(Y,Z)$.  By \eqref{eq::Kdecomp}, constancy of $\lambda$, \eqref{eq::leaf-connection}, and since $\widehat\nabla_Xe_1=-\cA X$ we obtain
\begin{align*}
\widehat\nabla_X(K(Y,Z))&={\bar\nabla}_X(\bar K(Y,Z))-\lambda h(Y,Z)\cA X\\
&\quad+\bigl(h(\cA X,\bar K(Y,Z))+\lambda X(h(Y,Z))\bigr)e_1.
\end{align*}
Moreover, since $K(e_1,Z)=\lambda Z$ and $K(Y,e_1)=\lambda Y$ we have
\begin{align*}
K(\widehat\nabla_XY,Z)&={\bar K}({\bar\nabla}_XY,Z)+\lambda h(\cA X,Y)Z+\lambda h({\bar\nabla}_XY,Z)e_1,
\end{align*}
and
\begin{align*}
K(Y,\widehat\nabla_XZ)&={\bar K}(Y,{\bar\nabla}_XZ)+\lambda h(\cA X,Z)Y+\lambda h(Y,{\bar\nabla}_XZ)e_1.
\end{align*}
Therefore, from the definition of the covariant derivative of $K$,
\begin{align}\label{eq::NablaK}
(\widehat\nabla_XK)(Y,Z)&=(\bar\nabla_X\bar K)(Y,Z)-\lambda\bigl(h(Y,Z)\cA X+h(\cA X,Y)Z+h(\cA X,Z)Y\bigr)\\
\notag &\quad+\Bigl[h(\cA X,\bar K(Y,Z))+\lambda\bigl(X(h(Y,Z))-h({\bar\nabla}_XY,Z)-h(Y,{\bar\nabla}_XZ)\bigr)\Bigr]e_1.
\end{align}
Since $\bar\nabla$ is the Levi-Civita connection of the induced metric,
$$
X(h(Y,Z))=h({\bar\nabla}_XY,Z)+h(Y,{\bar\nabla}_XZ),
$$
and the \eqref{eq::NablaK} reduces to
\begin{align*}
(\widehat\nabla_XK)(Y,Z)&=(\bar\nabla_X\bar K)(Y,Z)-\lambda\bigl(h(Y,Z)\cA X+h(\cA X,Y)Z+h(\cA X,Z)Y\bigr)\\
&\quad+h(\cA X,\bar K(Y,Z))e_1.
\end{align*}
Interchanging $X$ and $Y$ and using the Codazzi identity
$$
(\widehat\nabla_XK)(Y,Z)=(\widehat\nabla_YK)(X,Z),\qquad X,Y,Z\in T_2,
$$
we obtain for $T_2$ component:
\begin{align*}
&(\bar\nabla_X\bar K)(Y,Z)-(\bar\nabla_Y\bar K)(X,Z)\\
&\qquad=\lambda\bigl(h(Y,Z)\cA X-h(X,Z)\cA Y+h(\cA X,Y)Z-h(\cA Y,X)Z\\
&\hspace{46mm}+h(\cA X,Z)Y-h(\cA Y,Z)X\bigr).
\end{align*}
Since $h(\cA X,Y)=h(\cA Y,X)$ the terms proportional to $Z$ cancel and we get
\begin{align}\label{eq::tangentCodazzi}
&(\bar\nabla_X\bar K)(Y,Z)-(\bar\nabla_Y\bar K)(X,Z)\\
\notag &\qquad=\lambda\bigl(h(Y,Z)\cA X-h(X,Z)\cA Y+h(\cA X,Z)Y-h(\cA Y,Z)X\bigr).
\end{align}
The $e_1$-components of the same Codazzi identity give
\begin{equation}\label{eq::e1Codazzi}
h(\cA X,\bar K(Y,Z))=h(\cA Y,\bar K(X,Z)).
\end{equation}
Using the self-adjointness of $\cA$, the total symmetry of $h(\bar K(\cdot,\cdot),\cdot)$ and \eqref{eq::e1Codazzi} we obtain
\begin{align*}
h([\cA,\bar K_Z]X,W)&=h(\cA\bar K(Z,X),W)-h(\bar K(Z,\cA X),W)\\
&=h(\cA W,\bar K(X,Z))-h(\cA X,\bar K(W,Z))=0
\end{align*}
for every $W\in T_2$. Thus
$$
[\cA,\bar K_Z]=0\qquad\text{for every }Z\in T_2.
$$
In particular, we obtain \eqref{eq::A-K-commute} from Lemma~\ref{lm::CW}.
\par Let $e_2,\ldots,e_n$ be any local $h$-orthonormal frame of $TN$, and let $X\in\X(N)$.  Setting $Y=Z=e_i$ in \eqref{eq::tangentCodazzi} and summing over $i=2,\ldots,n$ we get
\begin{align}\label{eq::NablaKSum-ei}
&\sum_{i=2}^n(\bar\nabla_X\bar K)(e_i,e_i)-\sum_{i=2}^n(\bar\nabla_{e_i}\bar K)(X,e_i)\\
\notag &\qquad=\lambda\sum_{i=2}^n\bigl(h(e_i,e_i)\cA X-h(X,e_i)\cA e_i+h(\cA X,e_i)e_i-h(\cA e_i,e_i)X\bigr)\\
\notag &\qquad=\lambda\bigl((n-1)\cA X-\cA X+\cA X-(\tr\cA) X\bigr) = \lambda\bigl((n-1)\cA X-(\tr\cA) X\bigr).
\end{align}
In the above orthonormal basis we can write:
$$
\bar\nabla_Xe_i=\sum_{j=2}^n\omega_{ij}(X)e_j.
$$
Because the frame is $h$-orthonormal and $\bar\nabla h=0$ we have
$$
\omega_{ij}=-\omega_{ji}
$$
for $i,j=2,\ldots,n$.
Using \eqref{eq::barK-tracefree} we compute
\begin{align*}
0&=\bar\nabla_X\left(\sum_{i=2}^n\bar K(e_i,e_i)\right)\\
&=\sum_{i=2}^n(\bar\nabla_X\bar K)(e_i,e_i)+2\sum_{i=2}^n\bar K(\bar\nabla_Xe_i,e_i)\\
&=\sum_{i=2}^n(\bar\nabla_X\bar K)(e_i,e_i)+2\sum_{i,j=2}^n\omega_{ij}(X)\bar K(e_j,e_i).
\end{align*}
The last double sum vanishes because $\omega_{ij}=-\omega_{ji}$ and $\bar K(e_j,e_i)=\bar K(e_i,e_j)$.  Hence
$$
\sum_{i=2}^n(\bar\nabla_X\bar K)(e_i,e_i)=0.
$$
Since $\bar K$ is symmetric also $\bar\nabla_{e_i}\bar K$ is symmetric and we get
$$
\sum_{i=2}^n(\bar\nabla_{e_i}\bar K)(X,e_i)=\sum_{i=2}^n(\bar\nabla_{e_i}\bar K)(e_i,X).
$$

Now \eqref{eq::NablaKSum-ei} takes the form
\begin{equation}\label{eq::contracted-Kbar-application}
\sum_{i=2}^n(\bar\nabla_{e_i}\bar K)(e_i,X)=-\lambda\bigl((n-1)\cA-(\tr\cA)\id\bigr)X.
\end{equation}
By Lemma~\ref{lm::operatorA}
\begin{equation}\label{eq::A-scalar}
\cA=\frac{\tr\cA}{n-1}\id.
\end{equation}
Set
$$
b:=-\frac{\tr\cA}{n-1}.
$$
The identity \eqref{eq::A-second-form} gives
\begin{equation}\label{eq::umbilic}
h(\widehat\nabla_XY,e_1)=-b h(X,Y),\qquad\widehat\nabla_Xe_1=bX,\qquad X,Y\in T_2.
\end{equation}
Thus the leaves of $T_2$ are totally umbilical, with mean curvature vector $H=-be_1$. In consequence
$$
\widehat\nabla_XH=-X(b)e_1-b\widehat\nabla_X e_1=-X(b)e_1-b^2X.
$$
In order to prove that $T_2$ is spherical it is enough to show that $X(b)=0$ for $X\in T_2$.
Recall (see \cite[Chapter~II, \S9, formula~(9.2)]{NomSas}) that for an
affine hypersurface we have
$$
\widehat R(X,Y)Z=\frac{1}{2}\bigl\{h(Y,Z)SX-h(X,Z)SY+h(SY,Z)X-h(SX,Z)Y\bigr\}-[K_X,K_Y]Z
$$
where $S$ is the affine shape operator.
Since in our case $S=\epsc\id$, this simplifies to
$$
\widehat R(X,Y)Z=\epsc\bigl(h(Y,Z)X-h(X,Z)Y\bigr)-[K_X,K_Y]Z.
$$
Applying the above formula to $X,Y\in T_2$ and $e_1$ we obtain
$$
\widehat R(X,Y)e_1=-[K_X,K_Y]e_1=-\lambda\bigl(K(X,Y)-K(Y,X)\bigr)=0.
$$
On the other hand, \eqref{eq::umbilic} and the integrability of $T_2$ gives
$$
\widehat R(X,Y)e_1=X(b)Y-Y(b)X.
$$
Since $\dim T_2=n-1\ge2$,
\begin{equation}\label{eq::bZero}
X(b)=0\qquad\text{for every }X\in T_2.
\end{equation}
The distribution $T_1$ is integrable and totally geodesic, and $T_2$ is integrable and spherical thus Theorem~\ref{tw::Hiepko} gives the local warped product
$$
(M,h)=I\times_\rho N.
$$
\end{proof}
Now we can prove the main theorem of this paper.

\begin{proof}[Proof of Theorem~\ref{tw::main}]
If $\beta\neq2\alpha$, the result is exactly case~(1) of the
classification theorem of Cao and Wang \cite{CaoWang2026}.
\par Assume now that $\beta=2\alpha$. By
Proposition~\ref{stw::strengthenedProp4}, the warped-product structure
holds for every $n\ge3$, without any assumption on $\bar K$.
By Lemma~\ref{lm::CW}, \eqref{eq::umbilic}, and
\eqref{eq::bZero}, we have
$$
\mu=2\lambda,\qquad \lambda=\const,\qquad\widehat\nabla_{e_1}e_1=0,
$$
and, for $X\in T_2$,
$$
\widehat\nabla_Xe_1=bX,\qquad X(b)=0.
$$
Moreover, for $X,Y\in T_2$,
$$
\widehat\nabla_XY=\bar\nabla_XY-bh(X,Y)e_1,\qquad K(X,Y)=\bar K(X,Y)+\lambda h(X,Y)e_1,
$$
where $\bar K$ is trace-free.
\par We may choose a local coordinate $t$ on the one-dimensional factor such
that $e_1=\partial_t$. Since $\mu=2\lambda$ and $\lambda$ is constant,
$$
\lambda'=0=(\mu-2\lambda)b.
$$
The same curvature computation as in \cite{CaoWang2026} gives
$$
\widehat R(e_1,X)e_1=\bigl(e_1(b)+b^2\bigr)X=\bigl(\mu\lambda-\lambda^2-\epsc\bigr)X,
$$
for $X\in T_2$. Hence
$$
b'=-b^2-\epsc+\mu\lambda-\lambda^2.
$$
Set
$$
q:=\lambda^2-b^2-\epsc.
$$
Then
\begin{align*}
q'&=2\lambda\lambda'-2bb'\\
&=2b\bigl(\lambda\mu-2\lambda^2+b^2+\epsc-\mu\lambda+\lambda^2\bigr)=-2bq.
\end{align*}
Thus $q$ satisfies $q'=-2bq$. If $q$ vanishes at one point, the
uniqueness theorem gives $q\equiv0$ on the interval under
consideration. Otherwise, after shrinking the interval if necessary,
$q$ is nowhere zero. Therefore, locally,
$$
q\equiv 0 \qquad\text{or}\qquad q\neq 0.
$$
Choose positive functions $\kappa,\tau$ and a function $\phi$ satisfying
$$
\kappa'=-(b+\lambda)\kappa,\qquad\tau'=(\lambda+b-\mu)\tau,\qquad\phi'=\kappa\tau^{-1}.
$$
Following the construction from \cite{CaoWang2026} we define
$$
B_0=\tau\bigl(x_*(e_1)-(b+\lambda)x\bigr),\qquad A_0=\kappa x-\phi B_0.
$$
Using the Gauss formula and the relations above, we get
$$
D_{e_1}x_*(e_1)=\mu x_*(e_1)-\epsc x,
$$
$$
D_Xx_*(e_1)=(b+\lambda)x_*(X),\qquad X\in T_2
$$
and
$$
D_Xx_*(Y)=x_*\bigl(\bar\nabla_XY+\bar K(X,Y)\bigr)+\bigl((\lambda-b)x_*(e_1)-\epsc x\bigr)h(X,Y),\qquad X,Y\in T_2.
$$
The same computation as in \cite{CaoWang2026} gives
$$
D_XB_0=0\qquad\text{for every }X\in TM,\qquad D_{e_1}A_0=0.
$$
Thus $B_0$ is a constant vector and $A_0$ depends only on the variable
$p\in N$.
\par The remaining computations are the same as in the proof of
Proposition~5 of \cite{CaoWang2026}. We use the decomposition
$$
K(X,Y)=\bar K(X,Y)+\lambda h(X,Y)e_1
$$
and the trace-free property of $\bar K$ and obtain that $A_0$ defines a proper affine hypersphere
when $q\neq0$, and an improper affine hypersphere when $q\equiv0$.
\par After a centroaffine transformation, we may write the constant vector
$B_0$ as
$$
e=(0,\ldots,0,1)
$$
and the affine hypersphere $A_0$ as $g(p)$, where $g\colon N\to\R^n\times\{0\}$ is proper or
$g\colon N\to\{1\}\times\R^n$ is improper.
Finally we obtain
$$
x=\kappa^{-1}A_0+\kappa^{-1}\phi B_0=\gamma_1(t)e+\gamma_2(t)g(p),
$$
where
$$
\gamma_1=\kappa^{-1}\phi,\qquad\gamma_2=\kappa^{-1}.
$$
Together with the equations for $\kappa,\tau,\phi,\lambda$, and $b$
obtained above, this gives exactly the ODE system from
Theorem~\ref{tw::main}. This completes the proof.
\end{proof}

\section*{Acknowledgements}
This research was financed by the Ministry of Science and Higher Education of the Republic of Poland.

\end{document}